\documentclass[a4paper,11pt, twoside]{amsart}

\usepackage{amsmath,amsfonts,amsthm,amsopn,color,amssymb,enumitem}
\usepackage{palatino}
\usepackage[colorinlistoftodos]{todonotes}

\usepackage[colorlinks=true,urlcolor=blue, citecolor=red,linkcolor=blue,linktocpage,pdfpagelabels, bookmarksnumbered,bookmarksopen]{hyperref}
\usepackage[hyperpageref]{backref}

\usepackage[left=2.61cm,right=2.61cm,top=2.72cm,bottom=2.72cm]{geometry}
\newcommand{\e}{\varepsilon}

\newcommand{\R}{\mathbb{R}}

\newcommand{\RN}{{\mathbb{R}^N}}

 \DeclareMathOperator{\dv}{div}

\renewcommand{\le}{\leslant}
\renewcommand{\ge}{\geslant}
\renewcommand{\a }{\alpha }

\renewcommand{\d }{\delta }

\newcommand{\n }{\nabla }
\newcommand{\s }{\sigma }

\renewcommand{\H}{H^1(\RN)}

\def\bbm[#1]{\mbox{\boldmath $#1$}}
\newcommand{\beq }{\begin{equation}}
\newcommand{\eeq }{\end{equation}}

\renewcommand{\le}{\leqslant}
\renewcommand{\ge}{\geqslant}
\newcommand{\dis}{\displaystyle}

\providecommand{\pgfsyspdfmark}[3]{}

\makeatletter
\@addtoreset{equation}{section}%
\renewcommand{\theequation}{\thesection.\@arabic\c@equation}
\makeatother

\makeatletter
\providecommand\@dotsep{5}
\def\listtodoname{List of Todos}
\def\listoftodos{\@starttoc{tdo}\listtodoname}
\makeatother

\newtheorem{theorem}{Theorem}[section]

\newtheorem{remark}[theorem]{Remark}

\newtheorem{definition}[theorem]{Definition}

\title
[Oscillating solutions for prescribed mean curvature equations]
{Oscillating solutions for prescribed mean curvature equations: Euclidean and Lorentz-Minkowski cases}

\author[A. Pomponio]{Alessio Pomponio}

\address[A. Pomponio]{\newline\indent
Dipartimento di Meccanica, Matematica e Management
\newline\indent 
Politecnico di Bari
\newline\indent
Via Orabona 4,  70125  Bari, Italy}
\email{alessio.pomponio@poliba.it}

\thanks{}

\date{}

\begin{document}

\begin{abstract}
This paper deals with the prescribed mean curvature equations
\begin{equation*}
-\dv \left(\frac{\n u }{\sqrt{1\pm|\n u|^2}}\right)=g(u) \qquad \hbox{ in }\RN,
\end{equation*}
both in the  Euclidean case, with the sign ``$+$", and in the Lorentz-Minkowski case, with the sign ``$-$", 
for $N \ge 1$ under the assumption $g'(0)>0$. We show the existence of oscillating solutions, namely with an unbounded sequence of zeros. Moreover these solutions are periodic, if $N=1$, while they are radial symmetric and decay to zero at infinity with their derivatives, if $N\ge 2$. 
\end{abstract}

\maketitle

\section{Introduction}

Starting from the milestones papers \cite{BL,BLP,S}, the literature is plentiful of results  concerning the  following
class of nonlinear equations
\begin{equation}\label{delta}
-\Delta u=g(u) \qquad \hbox{ in }\RN,
\end{equation}
under the assumption $g'(0)<0$. In particular, for a large class of nonlinearities,  the existence of {\em ground state solutions}, namely radially symmetric positive solutions decaying at infinity, has been proved. Non-existence results of ground state solutions are present in \cite{NS1}. The so called {\em zero mass case}, that is when $g'(0)=0$, instead, has been object of study for example in \cite{BL,NS2}. The arguments in both the situations rely on variational techniques or on an ODE approach, such as the shooting method. Of course, we cannot mention all the contributes on this topic here.

Later on, different authors studied the existence and the non-existence of ground state solutions for a larger class of equations, replacing the Laplacian with different differential operators, under the assumption $g'(0)\le 0$. In particular, some results for the prescribed mean curvature equation in the Euclidean case
\begin{equation}\label{eq+}
-\dv \left(\frac{\n u }{\sqrt{1+|\n u|^2}}\right)=g(u) \qquad \hbox{ in }\RN,
\end{equation}  
and for some generalizations can be found, among others, in \cite{ADP,CG,dPG,FLS,FIN,KS,PS,PW}.

More recently, a lot of attention has been paid on the prescribed mean curvature equation in the Lorentz-Minkowski case
\begin{equation}\label{eq-}
-\dv \left(\frac{\n u }{\sqrt{1-|\n u|^2}}\right)=g(u) \qquad \hbox{ in }\RN.
\end{equation}
Existence and non-existence results of ground state solutions and of sign changing solutions, for the cases $g'(0)<0$ and  $g'(0)=0$ are contained in \cite{A,A2,BDD}. We remark that this kind of differential operator appears naturally also in the contest of Born-Infeld electro-magnetic theory, see  \cite{BDP,BI,BInat,FOP}.

 The case $g'(0)>0$, conversely, is completely different, indeed, as well explained in \cite{BL}, a direct consequence of this hypothesis is that radially symmetric $\H$ solutions of \eqref{delta} can not exist, and usual variational methods fail.  Nevertheless this case is very important since it is related to the study of the propagation of lights beams in a photorefractive crystals when a saturation effect is taken into account (see \cite{MMP} for a more precise description about these phenomena). Under this condition we can find, for example, the well known nonlinear Helmholtz equations. Contrary to the other cases, this one has not been studied intensively: some results  for \eqref{delta}, under several types of assumptions, can be found in \cite{E,EW1,EW2,EW3}. We mention in particular the recent paper \cite{MMP}, where the authors prove the existence of oscillating solutions (which are actually periodic for $N=1$) for \eqref{delta} with $N \ge 1$ and assuming that the nonlinearity $g$  is odd, with $g'(0) > 0$ and such that there exists $\a \in  (0, +\infty]$ such that $g$ is positive on $(0, \a)$ and negative on $(\a, +\infty)$.

Up to our knowledge, very little is known for problems \eqref{eq+} and \eqref{eq-} under these kind of assumptions on  $g$ and, in particular, there is no existence result of oscillating solutions. 
Aim of this paper, therefore, is to extend the results of \cite{MMP} to the prescribed mean curvature equations \eqref{eq+} and \eqref{eq-}, namely 
both in the  Euclidean case and in the Lorentz-Minkowski case,
for $N \ge 1$.

In the following, we will refer to a solution of each equation in \eqref{eq+} and \eqref{eq-} as a classical solution. More precisely, in the Euclidean case, $u$ is a solution of \eqref{eq+} if $u \in C^2(\RN)$ and satisfies the equation pointwise in $\RN$; in the Lorentz-Minkowski case, instead, $u$ is a solution of \eqref{eq-} if $u \in  C^2(\RN)$, $|\n u(x)| < 1$ for all $x \in  \RN $, and $u$ satisfies the equation pointwise in $\RN $.
Moreover, we need the following
\begin{definition}
A  solution $u$ of \eqref{eq+} or of \eqref{eq-} is called oscillating if it has an unbounded sequence of zeros. It is called localized when it converges to zero at infinity together with its partial derivatives up to order 2.
\end{definition}

In this paper, in the one-dimensional case, we will assume on the nonlinearity the following hypotheses:
\begin{enumerate}[label=(g\arabic{*}), ref=g\arabic{*}]
\item \label{g1} $g \in  C(\R)$;
\item \label{g2} $g$ is odd;
\item \label{g4} there exists $\a \in  (0, +\infty]$ such that $g$ is positive on $(0, \a)$ and negative on $(\a, +\infty)$.
\end{enumerate}
If $N\ge 2$, we will require in addition that
\begin{enumerate}[label=(g\arabic{*}), ref=g\arabic{*}]
\setcounter{enumi}{3}
\item \label{g3} $g$ is differentiable in $0$ and $g'(0) > 0$.
\end{enumerate}
Moreover, in the following we will denote by  $G(t)=\int_0^t g(s)\ ds$.
\medskip

For what concerns the Euclidean case, our main result is the following
\begin{theorem}\label{main+}
Assume \eqref{g1}-\eqref{g4} and, if $N\ge 2$, also \eqref{g3}. Then, for any $|\xi |<\a $ such that  $G(\xi)<1$, there exists an oscillating solution $u_\xi\in C^2(\RN)$ for \eqref{eq+} such that $u_\xi(0)=\xi$, $\|u_\xi\|_{L^\infty(\RN)}=|\xi|$. Moreover $u_\xi$ is periodic if $N=1$, while it is localized if $N\ge 2$.
\end{theorem}

In the Lorentz-Minkowski case, instead, our main result is the following
\begin{theorem}\label{main-}
Assume \eqref{g1}-\eqref{g4} and, if $N\ge 2$, also \eqref{g3}. Then, for any $|\xi |<\a $, there exists an oscillating solution $u_\xi\in C^2(\RN)$ for \eqref{eq-} such that $u_\xi(0)=\xi$, $\|u_\xi\|_{L^\infty(\RN)}=|\xi|$. Moreover $u_\xi$ is periodic if $N=1$, while it is localized if $N\ge 2$.
\end{theorem}

 If one deals with these two cases using variational techniques, the approaches and the functional settings are quite different. In the Lorentz-Minkowski case, for example, a variational approach to the problem can not be performed in the usual functional spaces. In particular, the quantity $1/\sqrt{1-|\n u(x)|^2}$ makes sense when $x\in\RN$ is such that $|\n u(x)|<1$, being this inequality a necessary constraint to be considered in the functional setting (see \cite{BDP}). However, since we are interested in radially symmetric solutions, we will see that our approach, based on ODE techniques, works well in both cases with only
some suitable modifications. 
\\
Our arguments are inspired by \cite{BLP,GZ,MMP} but are more involved and require some additional effort. For instance, if $u$ is a  solution of \eqref{eq+}, we don not know, in general, if its gradient is uniformly bounded, while in the Lorentz-Minkowski case, this uniform bound is obtained for free, since, for any $x\in \RN$, we have to require that $|\n u(x)|< 1$, but we have to be sure that $|\n u|$ remains  uniformly far away from $1$.
For this reason the assumptions of the two theorems are similar but not equal. In the Euclidean case, indeed, it is well known, at least for $N=1$, that problem \eqref{eq+} could have no classic solutions, but only bounded variation solutions, (see \cite{C} and the references therein). Therefore we have to add an additional assumption in order to avoid this particular situation.

The paper is organized as follows. In Section \ref{se:E} we deal with the Euclidean case and Theorem \ref{main+} will be an immediate consequence of Theorems \ref{n=1+} and \ref{n>1+}.  In Section \ref{se:LS}, instead, we treat the Lorentz-Minkowski case and Theorem \ref{main-}  will follow easily from Theorems \ref{n=1-} and \ref{n>1-}. Since the arguments are similar in both cases, in the last section we will skip some details underlying only the necessary differences.
We think that the theorems present in Sections \ref{se:E} and \ref{se:LS} themselves could be of interest.

\section{The Euclidean case}\label{se:E}

This section will be devoted to the Euclidean case.

Let us start with the one-dimensional case, where we can simply consider the following Cauchy problem 
\begin{equation}\label{eqr1+}
\begin{cases}
\dis -\left(\frac{u' }{\sqrt{1+(u')^2}}\right)' =g(u),  &\hbox{in }(0,+\infty),
\\[5mm]
u(0)=\xi ,
\quad
u'(0)=0,
\end{cases}
\end{equation}
where $\xi\in \R$.

The following result is partially already known, see for example \cite{C} and the references therein, but we present it for sake of completeness and because it is a crucial step for the study of the multi-dimensional case.

\begin{theorem}\label{n=1+}
Assume \eqref{g1}-\eqref{g4}. For any $\xi \in \R$ there exists a solution $u_\xi\in C^2([0,R_\xi))$ of the Cauchy problem (\ref{eqr1+}), where $R_\xi\in(0,+\infty]$ is such that $[0,R_\xi)$ is the maximal interval where the function $u_\xi$ is defined. Moreover, we have
\begin{itemize}
\item[(i)] if $|\xi|=\a \in \R$ or $\xi=0$, then $u_\xi\equiv \xi$;
\item[(ii)] if $\a \in \R$ and  $|\xi| > \a $, then $|u_\xi|$ strictly increases on $[0, R_\xi )$ and either $|u_\xi(r)|$ or $|u_\xi'(r)|$ diverges to $+\infty$, as $r\to  R_\xi^-$;
\item[(iii)] if $0 < |\xi| < \a$ and $G(\xi)\ge 1$, then $R_\xi\in \R$ and $|u_\xi'(r)|$ diverges to $+\infty$, as $r\to  R_\xi^-$, and  $\|u_\xi\|_{L^\infty([0,R_\xi))}=|\xi|$;
\item[(iv)] if $0 < |\xi| < \a$ and $G(\xi)<1$, then $R_\xi=+\infty$ and $u_\xi$ is oscillating and periodic with $\|u_\xi\|_{L^\infty(\R_+)}=|\xi|$.
\end{itemize}
\end{theorem}

\begin{proof}
By standard arguments,  (see for example \cite[Section 1.2]{C}), there exists a local solution $u_\xi$ of the (\ref{eqr1+}). Now let $R_\xi>0$ be such that $[0,R_\xi)$ is the maximal interval where the function $u_\xi$ is defined. We have $u_\xi\in C^2([0,R_\xi))$. In the following we simply write $u,R$ instead of $u_\xi,R_\xi$, respectively, for brevity. Moreover, being $g$ an odd function, by \eqref{g2}, we can reduce ourselves to consider only  the case $\xi\ge 0$. 
\medskip
\\
By the assumptions on $g$, {\it (i)} follows immediately. 
\medskip
\\
Let us prove {\it (ii)}, supposing that $\a \in \R$. We first prove that $u$ is strictly increasing in $[0,R)$. Since $u$ satisfies in $[0,R)$
\begin{equation}\label{1u''+}
-\frac{u'' }{(1+(u')^2)^\frac32}=g(u),
\end{equation}
we have that $u''(0)=-g(\xi)>0$. Therefore, defining $\bar r=\sup\{r\in [0,R):u'(r)>0\}$, we have that $\bar r\in (0,R]$. We have that $\bar r=R$, indeed, otherwise, if $\bar r<R$, by  \eqref{eqr1+} and \eqref{g4}, we would have
\[
\frac{u'(\bar r) }{\sqrt{1+(u'(\bar r))^2}} =-\int_0^{\bar r} g(u(s))\ ds>0,
\]
reaching a contradiction. Being $u$ strictly increasing in $[0,R)$, there exists $L=\lim_{r\to R^-}u(r)$. If $R\in \R$, then, by the maximality of $R$, we conclude that either $L=+\infty$ or $\lim_{r\to R^-}u'(r)=+\infty$. Let us consider the case $R=+\infty$. By \eqref{1u''+} and again by \eqref{g4}, we infer that $u$ is strictly convex on $[0,R)$ and we can conclude that $L=+\infty$.
\medskip
\\
Let us prove {\it (iii)} and {\it (iv)}. We start observing that multiplying equation of (\ref{eqr1+}) by $u'$ and integrating over $[0,r]$, we obtain the following equality for any $r\in (0,R)$
\begin{equation}\label{H1+}
H_+(u'(r)) =G(\xi)-G(u(r)),
\end{equation}
where $H_+(t)=\frac{\sqrt{1+t^2}-1}{\sqrt{1+t^2}}$. 
Since \eqref{H1+} is even with respect to $u$, by \eqref{g2}, and with respect to $u'$, it is standard to prove that $u$ is symmetric about critical points and antisymmetric about zeros. Therefore, it suffices to show that $u$ decreases until it attains a zero  in order to prove that $u$ is periodic and $\|u\|_{L^{\infty}}=\xi$.  
\\
By \eqref{eqr1+} and \eqref{g4}, for all $r>0$ such that $0<u<\a $ on $[0,r]$ we have 
\[
\frac{u'( r) }{\sqrt{1+(u'(r))^2}} =-\int_0^{ r} g(u(s))\ ds<0,
\]
and so $u$ is strictly  decreasing as long as it remains positive. 
Suppose that $u(r)>0$, for all $r\in [0,R)$, then by \eqref{1u''+} and since $g(u)>0$, we have that $u''(r)<0$ for all $r\in [0,R)$.
Being $u$ strictly positive, decreasing and concave, the possibilities are two: either $|u'|$ blows up at $R^-$ or $|u'|$ is uniformly bounded in $[0,R)$. In the former case, which happens, as observed in \cite[Section 1.2]{C}, whenever $G(\xi)\ge 1$, we deduce that $R\in\R$.  In the latter case, which occurs, at contrary, if and only if $G(\xi)<1$, we have that $R=+\infty$ reaching a contradiction: $u$ vanishes at some $r>0$ and, therefore, $u$ is periodic and oscillating. In both cases, we see that the $L^{\infty}$-norm of $u$ is~$\xi$, as desired.
\end{proof}

We pass now to consider the case $N\ge 2$.

Since we look for radial solutions, we can reduce equation \eqref{eq+} to the following Cauchy problem
\begin{equation}\label{eqr+}
\begin{cases}
\dis -\left(\frac{u' }{\sqrt{1+(u')^2}}\right)'-\frac{N-1}r\frac{u' }{\sqrt{1+(u')^2}}=g(u), &\hbox{in }(0,+\infty),
\\[5mm]
u(0)=\xi ,
\quad
u'(0)=0,
\end{cases}
\end{equation}
with $\xi\in \R$.

 We have the following

\begin{theorem}\label{n>1+}
Assume \eqref{g1}-\eqref{g3}. For any $\xi \in \R$ there exists a solution $u_\xi\in C^2([0,R_\xi))$ of the Cauchy problem (\ref{eqr+}), where $R_\xi>0$ is such that $[0,R_\xi)$ is the maximal interval where the function $u_\xi$ is defined. Moreover, we have
\begin{itemize}
\item[(i)] if $|\xi|=\a \in \R$ or $\xi=0$, then $u_\xi\equiv \xi$;
\item[(ii)] if $\a \in \R$ and  $|\xi| > \a $, then $|u_\xi|$ strictly increases on $[0, R_\xi )$ and either $|u_\xi(r)|$ or $|u_\xi'(r)|$ diverges to $+\infty$, as $r\to  R_\xi^-$;
\item[(iii)] 
if $0 < |\xi| < \a$ and $G(\xi)\le1$, then $R_\xi=+\infty$ and $u_\xi$ is oscillating and localized with $\|u_\xi\|_{L^\infty(\R_+)}=|\xi|$.
\end{itemize}
\end{theorem}

\begin{proof}
Also in this case, by \cite{NS2}, there exists a local solution $u_\xi$ of the Cauchy problem (\ref{eqr+}). Now let $R_\xi>0$ be such that $[0,R_\xi)$ is the maximal interval where the function $u_\xi$ is defined. We have $u_\xi\in C^2([0,R_\xi))$. In the following we simply write $u,R$ instead of $u_\xi,R_\xi$, for brevity. Moreover, being $g$ an odd function, by \eqref{g2}, we can reduce ourselves to consider only  the case $\xi\ge 0$. 
\medskip
\\
By the assumptions on $g$, {\it (i)} follows immediately. 
\medskip
\\
Let us prove {\it (ii)}, in the case of $\a \in \R$. We first prove that $u$ is strictly increasing in $[0,R)$. 
\\
Since
\[
\lim_{r\to 0^+}\frac{u'(r)}{r}=\lim_{r\to 0^+}\frac{u'(r)-u'(0)}{r}=u''(0),
\]
observing that $u$ satisfies in $(0,R)$
\begin{equation}\label{u''+}
-\frac{u'' }{(1+(u')^2)^\frac32}-\frac{N-1}r\frac{u' }{\sqrt{1+(u')^2}}=g(u),
\end{equation}
we have that $Nu''(0)=-g(\xi)>0$. Therefore, defining $\bar r=\sup\{r\in [0,R):u'(r)>0\}$, we deduce that $\bar r\in (0,R]$. We have that $\bar r=R$, indeed, otherwise, if $\bar r<R$, multiplying the equation in \eqref{eqr+} by $r^{N-1}$, integrating over $(0,\bar r)$, and by \eqref{g4}, we would have
\[
\frac{\bar r^{N-1} u'(\bar r) }{\sqrt{1+(u'(\bar r))^2}} =-\int_0^{\bar r}s^{N-1} g(u(s))\ ds>0,
\]
reaching a contradiction. Being $u$ strictly increasing in $[0,R)$, there exists $L=\lim_{r\to R^-}u(r)$. If $R\in \R$, then, by the maximality of $R$, we conclude that either $L=+\infty$ or $\lim_{r\to R^-}u'(r)=+\infty$ and we reach the conclusion. Let us consider, therefore, the case $R=+\infty$. Then, since $0 \le u'(r)/\sqrt{1+(u'(r))^2}\le 1$,  for all $r\ge 0$, we have
\[
\lim_{r\to +\infty}\frac{N-1}r\frac{u'(r) }{\sqrt{1+(u'(r))^2}}=0.
\]
Hence, by \eqref{u''+} and by \eqref{g4}, we infer that $u$ is strictly asymptotically convex  and so $L=+\infty$, as claimed.
\medskip
\\
Let us prove {\it (iii)}. Since the proof is quite long, we divide it into intermediate steps.
\\
{\it Step 1: $u$ decreases to a first zero}.
\\
For all $r>0$ such that $0<u<\a $ on $[0,r]$, multiplying the equation in  \eqref{eqr+} by $r^{N-1}$ and integrating over $(0,r)$, by \eqref{g4} we have 
\[
\frac{r^{N-1}u'( r) }{\sqrt{1+(u'(r))^2}} =-\int_0^{ r} s^{N-1}g(u(s))\ ds<0,
\]
and so $u$ is strictly decreasing as long as it remains positive. 
\\
Moreover, multiplying equation of (\ref{eqr+}) by $u'$ and integrating over $(0,r)$, we obtain the following equality, for any $r\in (0,R)$,
\begin{equation}\label{H+}
H_+(u'(r))+(N-1)\int_0^r \frac{(u'(s))^2}{s\sqrt{1+(u'(s))^2}}ds=G(\xi)-G(u(r)),
\end{equation}
where $H_+(t)=\frac{\sqrt{1+t^2}-1}{\sqrt{1+t^2}}$. Since $G(\xi)\le 1$, we deduce that $u'$ is bounded as long $u$ remains non-negative.
\\
Suppose by contradiction that $u(r)>0$, for all $r\in[0,R)$, then $R=+\infty$ and let $u_\infty=\lim_{r\to +\infty}u(r)\ge0$. 
\\
Let us consider, now, the function $Z_+:[0,+\infty)\to \R$, defined by 
\begin{equation}\label{z+}
Z_+(r) := 1-\frac1{\sqrt{1+(u'(r))^2}}+ G(u(r))
\end{equation}
and observe that $Z_+$ strictly decreases since, for all $r>0$, 
\begin{equation}\label{z'+}
Z_+'(r) =  \left[\left(\frac{u' }{\sqrt{1+(u')^2}}\right)'+g(u)\right]u'=-\frac{N-1}r\frac{(u')^2 }{\sqrt{1+(u')^2}}< 0.
\end{equation}
Therefore there exists $Z_\infty^+=\lim_{r\to +\infty}Z_+(r)$. 
\\
We show that, in this case, $\lim_{r\to +\infty}u'(r)=0$.
Indeed, otherwise, since $u$ admits a  horizontal asymptote, there would exist $c>0$ and two positive increasing diverging sequences $\{r_n\}_n$ and $\{s_n\}_n$, such that $s_n\in (r_n,r_{n+1})$, $u'(r_n)\le -c$, for $n$ sufficiently large, and $u'(s_n)\to 0$, as $n\to +\infty$.
Therefore we would have $Z_+(r_n)\ge 1-(1+c^2)^{-1/2}+G(u_\infty)$, for $n$ sufficiently large, while $Z_+(s_n)\to G(u_\infty)$, as $n\to +\infty$, contradicting the existence of the limit of $Z_+$ at infinity.
Hence, there exists $\bar c>1$ such that $1\le 1+( u'(r))^2\le \bar c$, for all $r\ge 0$. Therefore, by \eqref{eqr+}, we have, in $(0,+\infty)$,
\begin{equation*}
u''=-\frac{N -1}r u'[ 1+ (u')^2] -g(u) [1+( u')^2]^\frac32 \le  -\frac{(N -1)\bar c}{r}u' -g(u),
\end{equation*} 
where we have used the fact that $u' < 0$ and $g(u) > 0$, in $(0,+\infty)$. 
Therefore, if we set $v = r^\frac{(N-1)\bar c} 2 u$,  we get the following
\begin{equation}\label{sigma+}
v''+\s_2^+(r)v\le 0, \qquad\hbox{ where }\
\s_2^+(r):= \frac{g(u(r))}{u(r)}-\frac{(N -1)\bar c [(N -1)\bar c-2]}{4r^2}.
\end{equation}
By \eqref{g3}, we infer that there exists $c_0>0$ such that $\s_2^+(r)\ge c_0$, definitively, and so $v''$ is definitively negative. 
Being $v'$ definitively decreasing, there exists $\bar L=\lim_{r\to +\infty}v'(r)<+\infty$. Observe that $\bar L\ge 0$, because, otherwise, $\lim_{r\to +\infty}v(r)=-\infty$ contradicting the positivity of $u$. Hence $v$ is definitively increasing 
and then there exists $\bar r> 0$ such that, for any $r > \bar r$, we have $v(r) > v(\bar r)>0$. From \eqref{sigma+} we infer that, for some positive constant $C$, $v''(r) \le  -C < 0$ definitively. Integrating this last inequality over $(r_0,r)$, with a sufficiently large $r_0$, we deduce that $\bar L =  -\infty$ and we reach again a contradiction. 
%
%
Hence, $u$ attains a zero.
\\
{\it Step 2: $u$ is oscillating and  $\|u\|_{L^\infty(\R_+)}=\xi$}.
\\
Let us first show that there exist
$0 = r_0 < r_1 < r_2 < r_3 <\cdots$ such that all $r_{4j}$ are local maximizers, all $r_{4j+2}$
are local minimizers and all $r_{2j+1}$ are zeros of $u$. 
\\
The existence of a first zero $r_1 > 0 = r_0$ of $u$ has been shown in Step 1
and the strict monotonicity of $Z_+$ in $[r_0,r_1]$ implies $Z_+(r_1 ) < Z_+(r_0 )$. Concerning the
behaviour of $u$ on $[r_1 ,+\infty)$, there are now three possibilities:
\begin{itemize}
\item[(a)] $u$ decreases until it attains $-\xi$  at some $\bar r>r_1$;
\item[(b)] $u$ decreases on $[r_1 , +\infty)$ to some value $u_\infty \in  [-\xi, 0)$;
\item [(c)] $u$ decreases until it attains a critical point at some $r_2 > r_1$ with $-\xi < u(r_2 ) < 0$.
\end{itemize}
First of all, let us observe that, arguing as before, by \eqref{H+}, $u'$ remains bounded in all these three possibilities.
\\
Let us show that the cases (a) and (b) do not occur. Indeed, if there existed $\bar r > 0$ such that $u(\bar r) = -\xi$, then we would deduce that
\[
Z_+(\bar r) \ge G(u(\bar r)) = G(\xi)=G(u(0)) = Z_+(0)
\]
which is in contradiction with the strictly decreasing monotonicity of $Z_+$ in any interval whose extreme points are consecutive stationary points of $u$.  Hence, the case (a) is impossible. Let us now suppose that (b) holds. Arguing as in Step 1, we can prove that $\lim_{r\to +\infty}u'(r)=0$ and so, by \eqref{u''+}, we have $\lim_{r\to +\infty}u''(r)=-g(u_\infty)$. This implies that $g(u_\infty)$ must be equal to $0$. Being  $u_\infty \in  [-\xi, 0)$, we deduce that $u_\infty=-\xi$. Hence
\[
Z_+(r) \ge G(u(r))  \xrightarrow[r\to +\infty]{}   G(u_\infty)=G(\xi)=G(u(0)) = Z_+(0),
\]
reaching again a contradiction.
Therefore, we can say that $u$ decreases until it attains a critical point at some $r_2 > r_1$ with $-\xi < u(r_2 ) < 0$. 
Moreover, by \eqref{u''+}, \eqref{g2} and \eqref{g4}, we have
\[
u''(r_2)=-g(u(r_2))>0.
\]
Hence, $r_2$ is a local minimizer. By \eqref{g2} and \eqref{g4}, we can now
repeat the argument to get a zero $r_3 > r_2$, a local maximizer $r_4 > r_3$, a zero
$r_5 > r_4$ and so on, such that $\xi=u(r_0)>-u(r_2)>u(r_4)>\cdots$. Notice that this reasoning also shows that there are no further zeros or critical points. Moreover we conclude, also, that $\|u\|_{L^\infty(\R_+)}=\xi$.
\\
{\it Step 3: $u$ is localized}.
\\
First we show $u(r)\to 0$, as $r \to +\infty$. Take the sequence of maximizers $\{r_{4j}\}_j$ and assume, by contradiction, that $u(r_{4j})\to z \in (0,\xi)$, as $j\to +\infty$. Then \eqref{eqr+} and Ascoli-Arzel\`a Theorem imply that $\{u(\cdot + r_{4j})\}_j$ converges locally with respect to the $C^1$-norm to the unique solution $w$ of the Cauchy problem given by the equation in \eqref{eqr1+} and $w(0) = z,w'(0) = 0$, as $j\to +\infty$. Since $0\le G(z)<G(\xi)\le 1$, Theorem \ref{n=1+}-{\it (iv)} implies that $w$ is $T$-periodic with two zeroes in $[0,T]$ and let us call $\tau$ the first one. There exists $\d > 0$ such that $|w'| \ge  2\d$ on $[\tau-2\d,\tau+2\d]$.
Hence, for sufficiently large $j_0$ and for all $j\ge j_0$, we have 
\begin{equation}\label{210}
|u'(r_{4j}+r)|\ge \d , \quad \hbox{ for }r\in [\tau-\d,\tau+\d], \quad\hbox{ and }r_{4(j+1)}-r_{4j}\ge \tau-\d .
\end{equation}
So, for all $j\ge j_0$ , 
\[
\frac{(u'(r))^2}{\sqrt{1+(u'(r))^2}}\ge \d ' , \quad \hbox{ for }r\in [r_{4j}+\tau-\d,r_{4j}+\tau+\d], \quad\hbox{ where }\d'=\frac{\d ^2}{\sqrt{1+\d ^2}} .
\]
Finally, by \eqref{z'+}, for $k\ge j_0$ and $r\ge r_{4j}+\tau+\d$, we have 
\begin{align*}
Z_+(r) &=Z_+(0) -(N-1)\int_0^r\frac{(u'(s))^2 }{s\sqrt{1+(u'(s))^2}}\ ds
\\
&\le Z_+(0) -(N-1)\sum_{j=j_0}^{k}\int_{r_{4j}+\tau-\d}^{r_{4j}+\tau+\d}\frac{(u'(s))^2 }{s\sqrt{1+(u'(s))^2}}\ ds
\\
&\le Z_+(0) -\d '(N-1)\sum_{j=j_0}^{k}\int_{r_{4j}+\tau-\d}^{r_{4j}+\tau+\d}\frac{1 }{s}\ ds
\\
&= Z_+(0) -\d '(N-1) \sum_{j=j_0}^{k}\log \frac{r_{4j}+\tau+\d}{r_{4j}+\tau-\d}. 
\end{align*}
Now the arguments proceed as in \cite{GZ,MMP}, but we give some details for the sake of completeness.
\\
Let us fix $c(\d ')$ such that $\log(1+x)\ge c(\d ')x$ for all $0\le x\le 2\d /(r_{4j}+\tau-\d)$. Then by \eqref{210}, we deduce that
\begin{align*}
Z_+(r) &\le Z_+(0) -\d 'c(\d ')(N-1) \sum_{j=j_0}^{k} \frac{2\d}{r_{4j}+\tau-\d}
\\
&\le Z_+(0) -\d 'c(\d ')(N-1) \sum_{j=j_0}^{k} \frac{2\d}{r_{4j_0}+(j+1-j_0)\tau-\d}. 
\end{align*}
Since the harmonic series diverges, choosing $k$ and $r$ sufficiently large, we obtain that $Z_+(r)\to -\infty$ reaching a contradiction with the non-negativity of $Z_+$. As a consequence $u(r_{4j})\to 0$ and analogously we deduce that also $u(r_{4j+2})\to 0$. This implies that $u(r)\to 0$,  as $\to +\infty$.
\\
Being $\|u\|_{L^\infty(\R_+)}=\xi$, with $\xi\in (0,\a)$, by \eqref{g4}, $Z_+$ is non-negative; moreover, 
since $Z_+$ is decreasing, it follows that it admits a finite and non-negative limit at infinity. Hence, by \eqref{z+}, also $|u'|$ has a limit at infinity which must be zero because $u$ converges to $0$. Finally, from \eqref{u''+} we deduce that also $u''$ converges to zero at infinity.
\end{proof}

\begin{remark}
It is notable to observe that, if $0 < |\xi| < \a$ and $G(\xi)=1$, we obtain a different behavior for $N=1$ and for $N\ge 2$. Moreover, it would be interesting to understand what happens in the case if $0 < |\xi| < \a$ and $G(\xi)>1$ for the multi-dimensional case.
\end{remark}

\section{The Lorentz-Minkowski case}\label{se:LS}

This section will be devoted to the Lorentz-Minkowski case.

Let us start with the one-dimensional case, where we can simply consider the following Cauchy problem 
\begin{equation}\label{eqr1-}
\begin{cases}
\dis -\left(\frac{u' }{\sqrt{1-(u')^2}}\right)' =g(u), &\hbox{in }(0,+\infty),
\\[5mm]
u(0)=\xi ,
\quad
u'(0)=0,
\end{cases}
\end{equation}
where $\xi\in \R$.

As in the Euclidean case, the following result is partially already known, see for example \cite{C} and the references therein, but we present it for sake of completeness and because it is a crucial step for the study of the multi-dimensional case.

\begin{theorem}\label{n=1-}
Assume \eqref{g1}-\eqref{g4}. For any $\xi \in \R$ there exists a solution $u_\xi\in C^2([0,R_\xi))$ of the Cauchy problem (\ref{eqr1-}), where $R_\xi\in(0,+\infty]$ is such that $[0,R_\xi)$ is the maximal interval where the function $u_\xi$ is defined. Moreover, we have
\begin{itemize}
\item[(i)] if $|\xi|=\a \in \R$ or $\xi=0$, then $u_\xi\equiv \xi$;
\item[(ii)] if $\a \in \R$ and  $|\xi| > \a $, then $|u_\xi|$ strictly increases to $+\infty$ on $[0, R_\xi )$;
\item[(iii)] if $0 < |\xi| < \a$, then $R_\xi=+\infty$ and $u_\xi$ is oscillating and periodic with $\|u_\xi\|_{L^\infty(\R_+)}=|\xi|$.
\end{itemize}
\end{theorem}

\begin{proof}
By standard arguments (see for example \cite[Section 3.2]{C}), there exists a local solution $u_\xi$ of the (\ref{eqr1-}). Now let $R_\xi>0$ be such that $[0,R_\xi)$ is the maximal interval where the function $u_\xi$ is defined. We have $u_\xi\in C^2([0,R_\xi))$. In the following we simply write $u,R$ instead of $u_\xi,R_\xi$, respectively, for brevity. Moreover, being $g$ an odd function, by \eqref{g2}, we can reduce ourselves to consider only  the case $\xi\ge 0$. 
\\
We start observing that multiplying equation of (\ref{eqr1-}) by $u'$ and integrating over $(0,r)$, we obtain the following equality for any $r\in (0,R)$
\begin{equation}\label{H1-}
H_-(u'(r)) =G(\xi)-G(u(r)),
\end{equation}
where $H_-(t)=\frac{1-\sqrt{1-t^2}}{\sqrt{1-t^2}}$.
By \eqref{H1-}, we infer that $G(\xi)-G(u(r))\ge 0$, for any $r\in [0,R)$, and that $H_-(u'(r))$ is bounded if the right hand side is bounded: in particular, we have that
\begin{equation}\label{ce1}
\forall c>0, \exists\e=\e(c) >0: |u'(r)|\le 1-\e , \hbox{ whenever } |u(r)|\le c . 
\end{equation}
\\
By the assumptions on $g$, {\it (i)} follows immediately. 
\medskip
\\
Let us prove {\it (ii)}, assuming that $\a\in \R$. We first prove that $u$ is strictly increasing in $[0,R)$. Since $u$ satisfies in $[0,R)$
\begin{equation}\label{1u''-}
-\frac{u'' }{(1-(u')^2)^\frac32}=g(u),
\end{equation}
we have that $u''(0)=-g(\xi)>0$ and so, defining $\bar r=\sup\{r\in [0,R):u'(r)>0\}$, we deduce that $\bar r\in (0,R]$. We have that $\bar r=R$, indeed, otherwise, if $\bar r<R$, by  \eqref{eqr1-} and \eqref{g4}, we would have
\[
\frac{u'(\bar r) }{\sqrt{1-(u'(\bar r))^2}} =-\int_0^{\bar r} g(u(s))\ ds>0,
\]
reaching a contradiction. Hence, by \eqref{1u''-} and again by \eqref{g4}, we infer that $u$ is strictly convex on $[0,R)$ and we can conclude as in the Euclidean case.
\medskip
\\
Let us prove {\it (iii)}. 
Since \eqref{H1-} is even with respect to $u$, by \eqref{g2}, and to $u'$, then $u$ is symmetric about critical points and antisymmetric about zeros and so it is periodic. Therefore, it suffices to show that $u$ decreases until it attains a zero.  
\\
By \eqref{eqr1-} and \eqref{g4}, for all $r>0$ such that $0<u<\a $ on $[0,r]$ we have 
\[
\frac{u'( r) }{\sqrt{1-(u'(r))^2}} =-\int_0^{ r} g(u(s))\ ds<0,
\]
and so $u$ is decreasing as long as it remains positive. Suppose by contradiction that $u(r)>0$, for all $r>0$, then, by \eqref{1u''-} and \eqref{g4}, we have that $u''(r)<0$ for all $r>0$.  Being $u$ strictly positive, decreasing and concave, we reach immediately a contradiction. 
Hence, $u$ attains a zero and the proof is finished. 
%
%
\end{proof}

Let us now consider the multi-dimensional case.

If we look for radial solutions, we can reduce equation \eqref{eq-} to the following Cauchy problem
\begin{equation}\label{eqr-}
\begin{cases}
\dis -\left(\frac{u' }{\sqrt{1-(u')^2}}\right)'-\frac{N-1}r\frac{u' }{\sqrt{1-(u')^2}}=g(u),  &\hbox{in }(0,+\infty),
\\[5mm]
u(0)=\xi ,
\quad
u'(0)=0.
\end{cases}
\end{equation}

We have the following

\begin{theorem}\label{n>1-}
Assume \eqref{g1}-\eqref{g3}. For any $\xi \in \R$ there exists a solution $u_\xi\in C^2([0,R_\xi))$ of the Cauchy problem (\ref{eqr-}), where $R_\xi>0$ is such that $[0,R_\xi)$ is the maximal interval where the function $u_\xi$ is defined. Moreover, we have
\begin{itemize}
\item[(i)] if $|\xi|=\a \in \R$ or $\xi=0$, then $u_\xi\equiv \xi$;
\item[(ii)] if $\a \in \R$ and  $|\xi| > \a $, then $|u_\xi|$ strictly increases to $+\infty$ on $[0, R_\xi )$;
\item[(iii)] if $0 < |\xi| < \a$, then $R_\xi=+\infty$ and $u_\xi$ is oscillating and localized with $\|u_\xi\|_{L^\infty(\R_+)}=|\xi|$.
\end{itemize}
\end{theorem}

\begin{proof}
Also in this case, by \cite{NS2}, there exists a local solution $u_\xi$ of the Cauchy problem (\ref{eqr-}). Now let $R_\xi>0$ be such that $[0,R_\xi)$ is the maximal interval where the function $u_\xi$ is defined. We have $u_\xi\in C^2([0,R_\xi))$. In the following we simply write $u,R$ instead of $u_\xi,R_\xi$, for brevity. Moreover, being $g$ an odd function, by \eqref{g2}, we can reduce ourselves to consider only  the case $\xi\ge 0$.
%
%
%
%
\\
Multiplying equation of (\ref{eqr-}) by $u'$ and integrating over $(0,r)$, we obtain the following equality, for any $r\in (0,R)$,
\begin{equation}\label{H-}
H_-(u'(r))+(N-1)\int_0^r \frac{(u'(s))^2}{s\sqrt{1-(u'(s))^2}}ds=G(\xi)-G(u(r)),
\end{equation}
where $H_-(t)=\frac{1-\sqrt{1-t^2}}{\sqrt{1-t^2}}$.
By \eqref{H-} we infer that $G(\xi)-G(u(r))\ge 0$, for any $r\in (0,R)$, and that $H_-(u'(r))$ is bounded if the right hand side is bounded: in particular, we have that
\begin{equation}\label{ce}
\forall c>0, \exists\e=\e(c) >0: |u'(r)|\le 1-\e , \hbox{ whenever } |u(r)|\le c . 
\end{equation}
\\
By the assumptions on $g$, {\it (i)} follows immediately. 
\medskip
\\
Let us prove {\it (ii)}, in the case $\a \in \R$. We first prove that $u$ is strictly increasing in $[0,R)$. 
\\
Since
\[
\lim_{r\to 0^+}\frac{u'(r)}{r}=\lim_{r\to 0^+}\frac{u'(r)-u'(0)}{r}=u''(0),
\]
observing that $u$ satisfies
\begin{equation}\label{u''-}
-\frac{u'' }{(1-(u')^2)^\frac32}-\frac{N-1}r\frac{u' }{\sqrt{1-(u')^2}}=g(u),
\end{equation}
we have that $Nu''(0)=-g(\xi)>0$ and so, defining $\bar r=\sup\{r\in [0,R):u'(r)>0\}$, we deduce that  $\bar r\in (0,R]$. We have that $\bar r=R$, indeed, otherwise, if $\bar r<R$, multiplying the equation in \eqref{eqr-} by $r^{N-1}$, integrating over $(0,\bar r)$, by \eqref{g4} we would have
\[
\frac{\bar r^{N-1} u'(\bar r) }{\sqrt{1-(u'(\bar r))^2}} =-\int_0^{\bar r}s^{N-1} g(u(s))\ ds>0,
\]
reaching a contradiction. Being $u$ strictly increasing in $[0,R)$, there exists $L=\lim_{r\to R^-}u(r)$. If $R\in \R$, then, by the maximality of $R$, we conclude that $L=+\infty$. Let us consider the case $R=+\infty$ and assume by contradiction that $L\in \R$. Then, being $u$ bounded, by \eqref{ce} there exists $\d >0$ such that $\d \le \sqrt{1-(u'(r))^2}\le 1$,  for any $r\ge 0$, and so 
\[
\lim_{r\to +\infty}\frac{N-1}r\frac{u'(r) }{\sqrt{1-(u'(r))^2}}=0.
\]
Hence by \eqref{u''-} and by \eqref{g4}, we infer that $u$ is strictly convex definitively and so $L=+\infty$ reaching a contradiction.
\medskip
\\
Let us prove {\it (iii)}. As in the previous section, we divide the proof into intermediate steps.
\\
{\it Step 1: $u$ decreases to a first zero}.
\\
By \eqref{eqr-} and \eqref{g4}, for all $r>0$ such that $0<u<\a $ on $[0,r]$ we have 
\[
\frac{r^{N-1}u'( r) }{\sqrt{1-(u'(r))^2}} =-\int_0^{ r} s^{N-1}g(u(s))\ ds<0,
\]
and so $u$ is strictly decreasing as long as it remains positive. Suppose by contradiction that $u(r)>0$, for all $r>0$, then, being $u$ bounded, since by \eqref{ce} there exists $\d >0$ such that $\d \le \sqrt{1-(u'(r))^2}\le 1$, for any $r\ge 0$, by \eqref{u''-}, we have
\begin{equation*}
u''=-\frac{N -1}r u'[ 1- (u')^2] -g(u) [1-( u')^2]^\frac32 \le  -\frac{N -1}{r}u' -\d ^3g(u),
\end{equation*} 
where we have used the fact that $u'< 0$ and $g(u) > 0$. Therefore, if we set $v = r^\frac{N-1} 2 u$,  we get the following
\begin{equation*}\label{sigma-}
v''+\s_2^-(r)v\le 0 \qquad\hbox{ where }\
\s_2^-(r):= \frac{\d^3 g(u(r))}{u(r)}-\frac{(N -1)(N -3)}{4r^2}.
\end{equation*}
By \eqref{g3}, we infer that there exists $c_0>0$ such that $\s_2^-(r)\ge c_0$, definitively, and so $v''$ is definitively negative. Therefore, arguing as in the Euclidean case, we reach a contradiction.
\\
{\it Step 2: $u$ oscillates and  $\|u\|_{L^\infty(\R_+)}=\xi$}.
\\
Let us consider the function $Z_-:[0,+\infty)\to \R$, defined by
\begin{equation*}\label{z-}
Z_-(r) := \frac1{\sqrt{1-(u'(r))^2}}-1 + G(u(r))
\end{equation*}
and we observe that $Z_-$ decreases as
\begin{equation*}\label{z'-}
Z'_-(r) =  \left[\left(\frac{u' }{\sqrt{1-(u')^2}}\right)'+g(u)\right]u'=-\frac{N-1}r\frac{(u')^2 }{\sqrt{1-(u')^2}}\le 0. 
\end{equation*}
Moreover $Z_-'(r)=0$ if and only if $u'(r)=0$.
Arguing as in the Euclidean case, we show that there are
$0 = r_0 < r_1 < r_2 < r_3 <\cdots$ such that all $r_{4j}$ are local maximizers, all $r_{4j+2}$
are local minimizers and all $r_{2j+1}$ are zeros of $u$ and $\xi=u(r_0)>-u(r_2)>u(r_4)>\cdots$. Moreover there are no further zeros or critical points and $\|u\|_{L^\infty(\R_+)}=\xi$. 
\\
{\it Step 3: $u$ is localized}.
\\
First we show $u(r)\to 0$, as $r \to +\infty$. Take the sequence of maximizers $\{r_{4j}\}_j$ and assume by contradiction that $u(r_{4j})\to z \in (0,\xi)$, as $j\to +\infty$. Then \eqref{eqr-} and Ascoli-Arzel\`a Theorem imply that $\{u(\cdot + r_{4j})\}_j$ converges locally with respect to the $C^1$-norm to the unique solution $w$ of \eqref{eqr1-} with $w(0) = z,w'(0) = 0$, as $j\to +\infty$. Theorem \ref{n=1-}-{\it (iii)} implies that $w$ is $T$-periodic with two zeroes in $[0,T]$ and let us call $\tau$ the first one. There exists $\d \in (0,1)$ such that $|w'| \ge  2\d$ on $[\tau-2\d,\tau+2\d]$.
Hence, for sufficiently large $j_0$ and for all $j\ge j_0$, we have 
\[
|u'(r_{4j}+r)|\ge \d , \quad \hbox{ for }r\in [\tau-\d,\tau+\d], \quad\hbox{ and }r_{4(j+1)}-r_{4j}\ge \tau-\d .
\]
So, for all $j\ge j_0$, 
\[
\frac{(u'(r))^2}{\sqrt{1-(u'(r))^2}}\ge \d ' , \quad \hbox{ for }r\in [r_{4j}+\tau-\d,r_{4j}+\tau+\d], \quad\hbox{ where }\d'=\frac{\d ^2}{\sqrt{1-\d ^2}} .
\]
From now on, we can adapt easily the arguments of the Euclidean case to conclude.
\end{proof}

\subsection*{Acknowledgment}
The  author is partially supported by  a grant of the group GNAMPA of INdAM.  

The author wishes to express his more sincere gratitude to the anonymous referee: his/her acute comments and suggestions have been crucial to improve strongly the quality and the clarity of paper and to find and to fill a gap present in the previous version of this manuscript.


\begin{thebibliography}{99}
\bibitem{A}
A. Azzollini,  Ground state solution for a problem with mean curvature operator in Minkowski space. J. Funct. Anal. 266, 2086--2095, (2014).

\bibitem{A2}
A Azzollini, On a prescribed mean curvature equation in Lorentz-Minkowski space,
Journal de Math\'ematiques Pures et Appliqu\'ees, 106, 1122--1140.

\bibitem{ADP}
A. Azzollini, P. d'Avenia, A. Pomponio, Quasilinear elliptic equations in $\R^N$ via variational methods and Orlicz-Sobolev embeddings, Calc. Var. Partial Differential Equations, 49, (2014), 197--213. 

%
%
%

\bibitem{BL}
H. Berestycki, P.L. Lions, Nonlinear scalar field equations. I. Existence of a ground state, Arch. Ration. Mech. Anal., 82, (1983), 313--345.

\bibitem{BLP}
H. Berestycki, P.L. Lions, L.A. Peletier, An ODE approach to the existence of positive solutions for semilinear problems in $\RN$ , Indiana Univ. Math. J., 30, (1981), 141--157.

\bibitem{BDD}
D. Bonheure, A. Derlet, C. De Coster, Infinitely many radial solutions of a mean curvature equation in Lorentz– Minkowski space, Rend. Istit. Mat. Univ. Trieste, 44, (2012), 259--284.

\bibitem{BDP} 
D. Bonheure, P. d'Avenia, A. Pomponio, On the electrostatic Born-Infeld equation with extended charges, Comm. Math. Phys., 346, (2016), 877--906.

\bibitem{BInat}
M. Born, L. Infeld, Foundations of the new field theory, Nature, 132, (1933), 1004.

\bibitem{BI}
M. Born, L. Infeld, Foundations of the new field theory, Proc. Roy. Soc. London Ser. A, 144, (1934), 425--451.


\bibitem{CG}
M. Conti, F. Gazzola,
Existence of ground states and free-boundary problems for the prescribed mean-curvature equation,
Adv. Differential Equations, 7, (2002), 667--694.

\bibitem{C}
C. Corsato, Mathematical analysis of some differential models involving the Euclidean or the Minkowski mean curvature operator, Universit\`a degli Studi di Trieste, Trieste, 2015.

\bibitem{dPG}
M. del Pino, I. Guerra, Ground states of a prescribed mean curvature equation, J. Differential Equations, 241 (2007), 112--129.

\bibitem{E}
G. Evequoz, A dual approach in Orlicz spaces for the nonlinear Helmholtz equation, Z. Angew. Math. Phys., 66, (2015), 2995--3015, 2015.


\bibitem{EW1}
G. Evequoz, T. Weth, Branch continuation inside the essential spectrum for the nonlinear Schr\"odinger equation,  to appear on Journal of Fixed Point Theory and Applications.

\bibitem{EW2}
G. Evequoz, T. Weth, Real solutions to the nonlinear Helmholtz equation with local nonlinearity, Arch. Ration. Mech. Anal., 211, (2014), 359--388.

\bibitem{EW3}
G. Evequoz, T. Weth, Dual variational methods and nonvanishing for the nonlinear Helmholtz equation, Adv. Math., 280, (2015), 690--728.

\bibitem{FOP}
D. Fortunato, L. Orsina, L. Pisani,  Born-Infeld type equations for electrostatic fields, J. Math. Phys.,  43, (2002), 5698--5706.

\bibitem{FLS}
B. Franchi, E. Lanconelli, J. Serrin, Existence and Uniqueness of Nonnegative Solutions of Quasilinear Equations in $\R^n$, Adv. Math., 118, (1998), 177--243.


\bibitem{FIN}
N. Fukagai, M. Ito, K. Narukawa,
Positive solutions of quasilinear elliptic equations with critical Orlicz-Sobolev nonlinearity on $\RN$,
Funkcial. Ekvac., 49, (2006), 235--267.


\bibitem{GZ}
C. Gui, F, Zhou. Asymptotic behavior of oscillating radial solutions to certain nonlinear equations. Methods Appl. Anal., 15, (2008), 285--295.


\bibitem{KS}
T. Kusano, C.A. Swanson,
Radial entire solutions of a class of quasilinear elliptic equations,
J. Differential Equations, 83, (1990), 379--399.

\bibitem{MMP}
R. Mandel, E. Montefusco, B. Pellacci, Oscillating solutions for nonlinear Helmholtz Equations, preprint.

\bibitem{NS1}
W.-M. Ni, J. Serrin, Non-existence theorems for quasilinear partial differential equations,
Rend. Circ. Mat. Palermo, 5, (1986), 171--185.

\bibitem{NS2}
W.-M. Ni, J. Serrin, Existence and Non-existence theorems for quasi-linear partial differential equations. The anomalous case, Accad. Naz. Lincei, Convegni Dei Lincei, 77, (1986), 231--257.

\bibitem{PS}
L.A. Peletier, J. Serrin, Ground states for the prescribed mean curvature equation, Proc. Amer. Math. Soc., 100, (1987) 694--700.

\bibitem{PW}
A. Pomponio, T. Watanabe, Some quasilinear elliptic equations involving multiple $p$-Laplacians, to appear on Indiana Univ. Math. J.

\bibitem{S}
W.A. Strauss, Existence of solitary waves in higher dimensions, Comm. Math. Phys., 55, (1977), 149--162.
\end{thebibliography}
\end{document}